\documentclass[12pt,a4paper,leqno,twoside]{amsart}
\usepackage{latexsym}
\usepackage{amsmath}
\usepackage{amsfonts}
\usepackage{amssymb}

\numberwithin{equation}{section}

\theoremstyle{plain}
\newtheorem{thm}{Theorem}[section]
\newtheorem{cor}[thm]{Corollary}
\newtheorem{lem}[thm]{Lemma}
\newtheorem{prop}[thm]{Proposition}

\theoremstyle{definition}
\newtheorem{Def}[thm]{Definition}
\newtheorem{rem}[thm]{Remark}

\def\i{{\rm i}}
\def\j{{\rm j}}
\def\k{{\rm k}}
\def\Zz{{\mathbb Z}}
\def\Rr{{\mathbb R}}
\def\Cc{{\mathbb C}}
\def\Dd{\Delta_{\delta}}
\def\Cd{{\mathbb C}_{\delta}}
\def\Qq{{\mathbb Q}}
\def\Hh{{\mathbb H}}
\def\Hl{{\mathbb H}_{\lambda}}
\def\Ha{{\mathbb H}_{\alpha}}
\def\Hb{{\mathbb H}_{\beta}}

\def\Spinc{{\rm Spin}^c}
\def\Spin{{\rm Spin}}
\def\Hom{\operatorname{Hom}}
\def\End{\operatorname{End}}
\def\Aut{\operatorname{Aut}}
\def\Iso{\operatorname{Iso}}
\def\im{\operatorname{im}}
\def\op{{\rm o}}
\def\KO{K{\rm O}}

\def\tildeKO{\tilde K{\rm O}}
\def\KSp{K{\rm Sp}}
\def\K{K}
\def\SO{{\rm SO}}
\def\Sp{{\rm Sp}}
\def\O{{\rm O}}
\def\U{{\rm U}}
\def\SU{{\rm SU}}

\def\TSp{{\rm TSp}} 
\def\TU{{\rm TU}} 
\def\GL{{\rm GL}}

\def\Spin{{\rm Spin}}

\def\bar#1{\overline{#1}}

\def\da{d^{\alpha}}
\def\cd{c^{\delta}}
\def\BO{\operatorname{BrO}(X)}

\input xy
\xyoption{all}

\begin{document}

\title{Quaternionic structures}

\author[M.~\v Cadek, M.~C.~Crabb and J. Van\v zura]{Martin \v Cadek,
Michael Crabb, JI\v R\'I VAN\v ZURA}


\address{\newline Department of Mathematics, Masaryk University,
Kotl\' a\v rsk\' a 2, 611 37 Brno, Czech Republic}
\email{cadek@math.muni.cz}

\address{\newline Department of Mathematical Sciences, University
of Aberdeen, Aberdeen AB24 3UE, U.K.}
\email{m.crabb@maths.abdn.ac.uk}

\address{\newline Academy of Sciences of the Czech Republic, Institute
of Mathematics, \v Zi\v zkova 22,
616 62 Brno, Czech Republic}
\email{vanzura@ipm.cz}

\date{March 10, 2010}

\keywords{Bundles of quaternionic algebras; almost quaternionic manifolds;
vector bundles; characteristic classes; $K$-theory; Morita equivalence.}

\subjclass[2000]{55R25; 55R50; 55R40; 53C26}

\thanks{This work was supported by the the Czech Ministry of Education
[MSM0021622409]; the Academy of Sciences of the Czech Republic
[AVOZ10190503]; and the Grant Agency of the Czech Republic [201/05/2117].}

\abstract
{Any oriented $4$-dimensional real vector bundle is naturally
a line bundle over a bundle of quaternion algebras. In this paper 
we give an account of modules over bundles of quaternion algebras,
discussing Morita equivalence, characteristic classes and $K$-theory.
The results have been used to describe obstructions for the existence of almost 
quaternionic structures on $8$-dimensional Spin$^c$ manifolds in \cite{CCV} 
and may be of some interest, also, in quaternionic and algebraic geometry.}
\endabstract

\maketitle

\section{Introduction}

In this paper we will consider real and complex vector bundles which are
modules over bundles of quaternion algebras. Our study is motivated by
quaternionic geometry: the tangent bundles to quaternionic Kaehler 
and almost hyper-Kaehler manifolds provide examples of such bundles.
However, our methods are purely algebraic and topological.
Following classical ring theory we study Morita
equivalence of bundles of quaternion algebras. For a given
bundle of quaternion algebras we compute the Grothendieck group of left
modules as a classical $\KO$-group. For modules over a bundle 
of quaternion algebras we define 
characteristic classes and show that they behave as well as the Chern
classes for complex vector bundles. We also characterize those
complex vector bundles which are bundles over quaternion algebras.
In the final section we examine bundles of complexified quaternion 
algebras and their Morita equivalence. 

Our results extend the results obtained by Atiyah and Rees 
in \cite{AR} on complex quaternionic vector bundles and their $K$-theory and 
by Marchiafava and Romani (\cite{MR1, MR2, MR3}) concerning 
characteristic classes. Some of the results proved here were used in \cite{CCV} 
to derive the obstructions for the existence of almost quaternionic structures 
on 8-dimensional $\Spinc$ manifolds and in \cite{Sp} to compute indices of
elliptic complexes in quaternionic geometry. 

All bundles will be considered over a compact Hausdorff space $X$. Since the 
structure group of a bundle of quaternion algebras has to be $\Aut(\Hh)=\SO(3)$, 
every such bundle of quaternion algebras over $X$ is of the form 
$\Rr\oplus\alpha$ where $\Rr$ is a trivial real line bundle and $\alpha$ is 
an oriented 3-dimensional orthogonal vector bundle over $X$. 

\begin{Def}\label{Ha}
Let $\alpha$ be an oriented $3$-dimensional vector bundle over $X$
with a positive-definite inner product. We write $\Hh_\alpha$ for
the bundle of quaternion algebras $\Rr 1\oplus\alpha$ with the
multiplication
given in terms of the inner product $\langle -,-\rangle$ and
vector product $\times$ by
$$
(s,u)\cdot (t,v) = (st-\langle u,v\rangle , sv+tu+u\times v).
$$
Alternatively, thinking of
the group of automorphisms $\Aut (\Hh )$ of
$\Hh = \Rr 1\oplus\Rr\i\oplus\Rr\j\oplus\Rr\k$
as the special orthogonal group $\SO (3)$ of $\Rr^3=
\im\Hh$,
we have
$$
\Hh_\alpha = P\times_{\Aut (\Hh )} \Hh
\quad\text{and}\quad \alpha =P\times_{\SO (3)} \Rr^3,
$$
where $P\to X$ is the principal $\SO (3)=\Aut (\Hh )$-bundle given
by $\alpha$.
\end{Def}

It is clear that different inner products on $\alpha$ define isomorphic
bundles of quaternion algebras. Conjugation $^{-} : \Hh \to \Hh$ gives an
isomorphism from $\Hh$ to
the opposite algebra $\Hh^\op$ compatible with the action of $\Aut (\Hh )$.
Hence
$\Ha$ and $\Ha^\op$ are isomorphic as bundles of algebras.
Notice that $^{-} :\alpha \to \alpha$ (that is, $-1$) is an orientation-reversing
isomorphism. 

\begin{Def}\label{Ha-bundle}
Let $\xi$ be a real vector bundle over $X$. We say that $\xi$  is
an {\it $\Ha$-bundle} if it has a left $\Ha$-module
structure, that is, a bundle map $\Ha\otimes_\Rr \xi \to \xi$
that restricts to a module structure in each fibre.
\end{Def}

\begin{rem}\label{orientation}
An $\Ha$-bundle $\xi$ has a canonical orientation. To orient the fibre $\xi_x$ 
at $x\in X$ we choose a basis $e_1,\ldots ,e_n$ as 
$\Hh_{\alpha_x}$-vector space and choose an oriented orthonormal basis 
$\i ,\, \j ,\, \k$ of $\alpha_x$. Then $e_1$, $\i e_1$, $\j e_1$, $\k e_1,
\ldots ,\, e_n$, $\i e_n$, $\j e_n$, $\k e_n$ orients $\xi_x$.
In particular, the orientation of $\Ha$ is determined by the orientation of 
$\alpha$. 
\end{rem}

\begin{rem}\label{metric}
If $\xi$ has an $\Ha$-structure, we may
(using a partition of unity to glue together local metrics)
choose a positive-definite real
inner product on $\xi$ such that the structure homomorphism
$$
\rho :\Ha \to \End_\Rr (\xi )
$$
is a $*$-homomorphism, that is, $\rho (r)^*=\rho (\bar{r})$ for
$r\in \Ha$.
\end{rem}

Given a $4n$-dimensional real inner product space $V$ with a left
$\Hh$-module structure compatible, as above, with the inner product,
we write $\O(V)$ and $\Sp(V)$ for the orthogonal and symplectic groups of
$V$ and define
$$
\TSp (V) = \{ g\in \O (V) \mid g(rv) = \kappa (r)g(v)
\text{\ for some $\kappa\in\Aut (\Hh )$ and  all $v\in V$, $r\in\Hh$}\}.
$$
(The `T' is intended to indicate `twisted'.)
It is evidently a subgroup of the special orthogonal group
$\SO(V)$ and we have an extension
$$
1 \to \Sp (V) \to \TSp (V) \to \Aut (\Hh ) \to 1.
$$
Since automorphisms of $\Hh$ are inner, we may equivalently describe
$\TSp (V)$ as the subgroup $\Sp (1)\cdot\Sp (V)
=(\Sp(1)\times \Sp(V))/\{\pm (1,1)\}$ of orthogonal maps of
the form $v\mapsto a\cdot g(v)$, where $a\in \Sp (1)\subseteq\Hh$ and
$g\in\Sp (V)$.
The group $\TSp (V)$ acts (orthogonally) on $\Hh$ and on $V$.

\begin{lem}\label{TSp}
Let $\xi$ be a $4n$-dimensional orthogonal real vector bundle.
Then $\xi$ admits an $\Hh_\alpha$-structure for some $\alpha$ if
and only if the structure group of
$\xi$ reduces from $\O (\Hh^n)$ to $\TSp (\Hh^n)$.
\end{lem}

\begin{proof}
If the structure group of $\xi$ reduces from $\O(\Hh^n)$ to $\TSp(\Hh^n)$,
there is a principal $\TSp (\Hh^n)$-bundle $P\to X$ such that
$\xi = P\times_{\TSp (\Hh^n)} \Hh^n$. Then we have an oriented orthogonal 
$3$-dimensional vector bundle $\alpha = P\times_{\TSp (\Hh^n)}\im\Hh$,
with an associated quaternion algebra $\Ha= P\times_{\TSp (\Hh^n)} \Hh$
having an obvious left action on $\xi$. The real vector space $\Hh^n$ has a
canonical orientation as a left $\Hh$-module. Then the choice of a
$\TSp(\Hh^n)$-principal bundle gives orientations to the bundles $\xi$ 
and $\alpha$ such that the orientation of $\xi$ is canonical with respect to
$\Ha$-structure.

If $\xi$ admits an $\Ha$-structure for some $\alpha$, then the bundle of
frames
\begin{align*}
\operatorname{Fr}(\xi)=\{&f\in\Hom(\Hh^n,\xi) \mid
\text{$f$ is a real isometry, $f(rv)=\rho(r)f(v)$}\\
&\text{for some $\rho\in \Iso(\Hh,\Ha)$ and all $v\in\Hh^n$, $r\in \Hh$}\}.
\end{align*}
is a principal $\TSp(\Hh^n)$-bundle. As above,
$\xi=\operatorname{Fr}(\xi)\times_{\TSp(\Hh^n)}\Hh^n$ and 
$\Ha=\operatorname{Fr}(\xi)\times_{\TSp(\Hh^n)}\Hh$.
Hence the structure group of $\xi$ reduces from $\O(\Hh^n)$ to $\TSp(\Hh^n)$.
\end{proof}

\begin{rem}\label{twist1} 
In quaternionic geometry (\cite{S1, S2}) a smooth $4n$-manifold $M$
is said to be
almost quaternionic if its tangent bundle is associated to a principal 
$\GL(n,\Hh)\cdot \Sp(1)$-bundle. After a choice of compatible metric this
principal bundle reduces to a principal $\TSp(n)$-bundle $P$. The sphere bundle 
$S(\alpha)$ of the vector bundle 
$\alpha=P\times_{\TSp(n)}\im\Hh$ is called the twistor space of $M$.  

Considering $\Hh$ as a $2$-dimensional complex vector space with multiplication 
by complex numbers from the right and with the standard left action of $\Sp(1)$,  
denote by $S^2_{\Cc}\Hh$ the second symmetric power with the induced action of
$\Sp(1)$. One can show that $\alpha\otimes\Cc$ is isomorphic to the complex 
vector bundle $S^2H=P\times_{\TSp(n)}S^2\Hh$. The isomorphism 
is induced by the $Sp(1)$-invariant homomorphism of real representations
$\varphi:\operatorname{im} \Hh\to S^2\Hh$, $\varphi(u)=\j\otimes u -1\otimes
u\j$. Formally, $S^2H$ is the second
symmetric power of a vector bundle $H$. Such a bundle $H$ exists globally
if $w_2(\alpha)=0$. Then $\alpha$ is associated to a principal 
$\Spin(3)=\SU(2)$-bundle $Q$, $H=Q\times_{\SU(2)}\Hh$, and the twistor space 
is the complex projective bundle $\Cc P(H)$
(\cite{S1}). 

As a specific example we have the quaternionic projective space $\Hh P(\Hh^{n+1})$
where $\alpha$ is the Lie algebra bundle.
\end{rem}

For any $\Ha$-vector bundle $\xi$ one can form the associated projective bundle,
which we denote by $\Ha P(\xi)$.

\section{Quaternionic line bundles and Morita equivalence}

Since $\TSp(\Hh)=\SO (4)$, every oriented 4-dimensional real vector
bundle $\mu$ has an $\Ha$-structure
for some $\alpha$. 
In this section we describe all such
structures and their properties and define the notion of 
Morita equivalence of bundles of quaternion algebras.

Recall the double covers
$$
\Sp (1) \times \Sp (1) \to \SO (4)=\SO (\Hh ) \xrightarrow{(\rho_+\,,\rho_-\,)} \SO (3)\times \SO (3)
$$
given by mapping $(a,b)\in \Sp (1)\times \Sp (1)$ to the
map $g: v\mapsto av\bar{b}$ in $\SO (\Hh )$ and
$g$ to  $(\rho_+(g),\rho_-(g))=(\rho (a),\rho (b))$,
where $\rho : \Sp (1) \to \SO (3)$ maps $a$ to $v\mapsto av\bar{a}$
(the adjoint representation).

This leads to a complete description of the twisted quaternionic line bundles.

\begin{prop}\label{Ha_Hb}
Let $\mu$ be an oriented $4$-dimensional orthogonal vector bundle
over $X$.
Write $\alpha =\rho_+(\mu )$ and $\beta =\rho_-(\mu )$. Then
$\mu$ is an $\Hh_\alpha$-line bundle and a right $\Hb$-line bundle, and there
is a canonical isomorphism (of bundles of algebras)
$$
\Hh_\alpha\otimes_\Rr \Hh_\beta^\op = \End_{\Rr} (\mu ).
$$

Conversely, if $\alpha$ is an oriented orthogonal $3$-dimensional vector
bundle
and $\mu$ is an orthogonal $\Hh_\alpha$-line bundle, then $\mu$ acquires
an orientation under which $\rho_+(\mu )$ is identified with
$\alpha$ and $\beta =\rho_-(\mu )$ is characterized by an
isomorphism (of bundles of algebras)
$$
\Hh_\beta^\op = \End_{\Hh_\alpha} (\mu ).
$$
Moreover, we have
$$
w_2(\alpha )=w_2(\beta )=w_2(\mu ).
$$
\end{prop}

\begin{proof}
Let $\mu$ be an oriented 4-dimensional orthogonal vector bundle. There is a
principal $\SO(\Hh)=\TSp(1)$-bundle $P$ such that
$\mu=P\times_{\SO(\Hh)}\Hh$.
From the definition, $\Ha=P\times_{\rho_+}\Hh$ and  $\Hb=P\times_{\rho_-}\Hh$,
where $\rho_+$ and $\rho_-$ determine the respective actions of $\TSp(1)$ on $\Hh$.
Then $\mu$ is a left $\Ha$-line and a right $\Hb$-line.
Next
$$\End_{\Rr}(\mu)=P\times_{\SO(\Hh)}\End_{\Rr}(\Hh),$$
where the action of $\SO(\Hh)$ on $\End_{\Rr}(\Hh)$ is
$(a,b)\cdot f:v\mapsto af(\bar{a}vb)\bar{b}$.
Similarly,
$$\Ha\otimes_{\Rr}\Hb^\op=P\times_{\rho_+\otimes\rho_-}
(\Hh\otimes_{\Rr}\Hh^\op),$$
where $\rho_+\otimes\rho_-$ acts on $\Hh\otimes\Hh^\op$  by
$(a,b)\cdot (h_1\otimes h_2)=ah_1\bar{a}\otimes bh_2\bar{b}$.
Since the isomorphism of algebras
$\Hh\otimes_{\Rr}\Hh^\op\to\End_{\Rr}(\Hh)$
given by
$h_1\otimes h_2\mapsto (v\mapsto h_1vh_2)$
is invariant under the actions of $\SO(\Hh)$ described above,
we get $\Ha\otimes\Hb^\op=\End_{\Rr}(\mu)$.  

Conversely, suppose that $\mu$ is an orthogonal $\Ha$-line. On $\mu$
consider the
canonical orientation given by the $\Ha$-structure. As in the proof of Lemma
\ref{TSp} we can construct a principal $\TSp(1)$-bundle
$P$ such that $\mu=P\times_{\TSp(1)}\Hh$ and $\Ha=P\times_{\rho_+}\Hh$.
(This gives $\mu$ the canonical orientation with respect to $\Ha$.)
Hence $\alpha=\rho_+(\mu)$.

Further, we have the monomorphism of algebras
$\Hb^\op\hookrightarrow \Ha\otimes_{\Rr}\Hb^\op=\End_{\Rr}(\mu)$.
Its image is $\End_{\Ha}(\mu)$. Consequently, $\Hb^\op=\End_{\Ha}(\mu)$.

Without loss of generality we can assume that $X$ is a finite CW-complex.
Every $4$-dimensional real vector bundle $\mu$ over the $3$-skeleton of $X$
has a section. If it is an $\Ha$-line and a right $\Hb$-line, then over the
$3$-skeleton $\mu=\Ha=\Rr\oplus\alpha$ and $\mu=\Hb^\op=\Rr\oplus\beta$.
Consequently, $w_2(\mu)=w_2(\alpha)=w_2(\beta)$.
\end{proof}

\begin{lem}\label{lambda2}
The vector bundles $\alpha$ and $\beta$ defined in the previous proposition
are isomorphic to the bundles of eigenspaces of the Hodge star
operator on $\Lambda^2\mu$ corresponding to the eigenvalues $1$ and $-1$,
respectively,
and consequently
$$\alpha\oplus\beta=\Lambda^2\mu.$$
\end{lem}

\begin{proof}
On $\Hh$ consider the real inner product $\langle u,v\rangle=
\operatorname{Re}(u\bar{v})$. The inner product enables us
to define the isomorphism of real vector spaces
$\Hh\otimes\Hh\cong \End_{\Rr}(\Hh): h_1\otimes h_2\mapsto (v\mapsto \langle
h_1, v\rangle h_2)$. Under this isomorphism the space $\Lambda^2\Hh$ maps
onto
the subspace of all skew-symmetric endomorphisms on $\Hh$. Since the left
multiplication $L_h$ and the right multiplication $R_h$ by a pure imaginary
quaternion $h$ are skew-symmetric endomorphisms of $\Hh$, we get two
inclusions $\Lambda^+:\im\Hh\to\Lambda^2\Hh$ and
$\Lambda^-:\im\Hh\to\Lambda^2
\Hh$ which are isometries (up to a multiple). It is easy to show that
\begin{equation*}
\Lambda^+(h)=1\wedge h+\i\wedge h\i+\j\wedge h\j+\k\wedge h\k,
\quad
\Lambda^-(h)=1\wedge h+\i\wedge\i h+\j\wedge\j h+\k\wedge\k h.
\end{equation*}
Put $\omega=1\wedge\i\wedge\j\wedge\k$. Then one can check that
$$\Lambda^+(h_1)\wedge\Lambda^+(h_2)=8\langle h_1,h_2\rangle \omega,\quad
\Lambda^-(h_1)\wedge\Lambda^-(h_2)=-8\langle h_1,h_2\rangle \omega,\quad
\Lambda^+(h_1)\wedge\Lambda^-(h_2)=0,$$
which means that $\Lambda^+(\im\Hh)$ and $\Lambda^-(\im\Hh)$ correspond
to the perpendicular eigenspaces for eigenvalues $1$ and $-1$, respectively, 
of the Hodge star operator. Hence we get an isomorphism $\Lambda^+(\im\Hh)\oplus
\Lambda^-(\im\Hh)=\Lambda^2\Hh$.

We have associated vector bundles
$\Lambda^2\mu=P\times_{\SO(\Hh)}\Lambda^2\Hh$,
$\alpha=\rho_+(\mu)=P\times_{\rho_+}\im\Hh$,
$\beta=\rho_-(\mu)=P\times_{\rho_-}\im\Hh$. Since the inclusions
$\Lambda^+$, $\Lambda^-$ are invariant with respect to actions of $\SO(\Hh)$
given on $\im \Hh$ by $\rho_+$ and $\rho_-$, respectively, we may define
$\Lambda^+:\alpha\to\Lambda^2\mu$, $\Lambda^-:\beta\to\Lambda^2\mu$.
The considerations above imply that
$\Lambda^+(\alpha)\oplus\Lambda^-(\beta)=\Lambda^2\mu$.
\end{proof}

\begin{cor}\label{2cor} 
Every $4$-dimensional oriented orthogonal vector bundle $\mu$ is a (left) 
module over just two (up to orientation preserving isomorphism) quaternion 
algebras, namely $\Ha$ and $\Hb$ where $\alpha=\rho_+(\mu)$ and 
$\beta=\rho_-(\mu)$. The orientation of $\mu$ is canonical with respect to 
the $\Ha$-structure and is not canonical with respect to the $\Hb$-structure. 
If $-\mu$ is the same vector bundle with opposite orientation then
$$\rho_+(-\mu)=\rho_-(\mu), \quad \rho_-(-\mu)=\rho_+(\mu).$$ 
\end{cor}

\begin{proof} An orientation of $\mu$ is determined by the choice of the
$\SO(\Hh)$-principal bundle $P$ such that $\mu=P\times_{\SO(\Hh)}\Hh$
and by the standard orientation of $\Hh$.
The standard orientation of $\im\Hh$ gives an orientation to 
$\alpha=P\times_{\rho_+}\im \Hh$ and $\beta=P\times_{\rho_-}\im\Hh$. 
Hence $\mu$ has the canonical orientation as an $\Ha$-module  
which coincides with the canonical orientation of $\mu$ as 
a right $\Hb$-module (given by the multiplication by $\i$, $\j$,
$\k$ from the right). Using the previous proposition  there are no 
other (up to orientation preserving isomorphisms) left and right quaternion
structures on $\mu$ with compatible orientations.
Since the left action of $\Hb$ is given by the right action of $\Hb$ and by the 
orientation reversing conjugation $^-:\Hb\to \Hb$, the orientation of $\mu$ is 
not canonical with respect to the left $\Hb$-structure. This implies the 
formulas.
\end{proof}    

\begin{lem}\label{p1}
The first Pontryagin classes of $\alpha$ and $\beta$
are
$$
p_1(\alpha)=p_1(\mu)+2e(\mu), \quad p_1(\beta)=p_1(\mu)-2e(\mu).
$$
\end{lem}

\begin{proof} Let $H$ be the Hopf $\Ha$-line bundle over the  
quaternionic projective bundle $\Ha P(\Ha^{\infty})$.
Since any $\Ha$-line bundle $\mu$ is a summand of a trivial $\Ha$-bundle
over $X$, there is a section $s:X\to \Ha
P(\Ha^{\infty})$ such that $\mu=s^*(H)$. So it is enough to check the
formulas for the bundle $H$. By the Leray-Hirsch Theorem, $H^*(\Ha
P(\Ha^{\infty}))=H^*(X)[e]$, where $e$ is the Euler class of $H$.  Further, the 
pullback of $H$ to the sphere bundle $S(\Ha^{\infty})$ is $\Ha$. So we have
$p_1(H)=p_1(\Ha)+ae$, where $a\in\Zz$. Restricting $H$ to a point in $X$ we
get Hopf bundle over $\Hh P^{\infty}$ and determine that $a=-2$.
\end{proof}

\begin{Def}\label{Morita}
Two bundles $\Hh_\alpha$ and $\Hh_\beta$ are said to be
{\it Morita equivalent} if
there is a $4$-dimensional real vector bundle $\mu$ and
an isomorphism of bundles of algebras
$$
\Hh_\alpha \otimes_\Rr \Hb^\op\cong \End_\Rr(\mu).
$$
Such an isomorphism defines a morphism in the following category
$\mathcal{M}(X)$. The objects of $\mathcal{M}(X)$ are the oriented
$3$-dimensional vector bundles over $X$. A morphism $\beta\to\alpha$
is represented by an isomorphism 
$\Hh_\alpha \otimes_\Rr \Hb^\op\to \End_\Rr(\mu)$,
two such isomorphisms $\Hh_\alpha \otimes \Hh_\beta^\op \to
\End_\Rr(\mu)$
and
$\Hh_\alpha \otimes \Hh_\beta^\op \to \End_\Rr(\mu')$
being regarded as the same if $\mu$ and $\mu'$ are isomorphic
as $\Hh_\alpha\otimes\Hh_\beta^\op$-modules. The identity $\alpha\to\alpha$
is given by $\mu=\Ha$. Composition is the tensor 
product: $\Hh_\alpha \otimes \Hh_\beta^\op \to \End_\Rr(\mu)$
and
$\Hh_\beta \otimes \Hh_\gamma^\op \to \End_\Rr(\nu)$
compose to $\Hh_\alpha\otimes\Hh_\gamma^\op\to \End_\Rr(\mu\otimes_{\Hh_\beta}
\nu )$. The category $\mathcal{M}(X)$ is a groupoid with
the inverse of $\Hh_\alpha\otimes\Hh_\beta^\op\to\End_\Rr(\mu)$ 
given by the dual $\Hh_\beta\otimes\Hh_\alpha^\op\to \End_\Rr(\mu^*)$
(or by $\Hh_\beta\otimes\Hh_\alpha^o\to \End_\Rr(\mu)$, using the
isomorphisms $\Hh_\alpha\to\Hh_\alpha^o$ and $\Hh_\beta\to \Hh_\beta^o$
and the inner product on $\mu$).
\end{Def}

Morita equivalence is usually defined as an equivalence of categories of
left modules. In our case both definitions are equivalent (in Proposition
\ref{Morita_eq} we prove one direction) and we have chosen the one which is
more suited to our purposes.

\begin{prop}\label{Morita_eq}
An isomorphism $\Ha\otimes\Hb^\op\cong\End_\Rr(\mu)$ determines an
equivalence from the category of $\Hb$-modules to the category of $\Ha$-modules
given by $\eta\mapsto\mu\otimes_{\Hb}\eta$.
\end{prop}

\begin{proof}
This is a consequence of the fact that $\mathcal{M}(X)$ is a groupoid.
\end{proof}

There is a topological criterion for Morita equivalence:

\begin{thm}\label{w2}
Two bundles $\Hh_\alpha$ and $\Hh_\beta$
are Morita equivalent if and only if $w_2(\alpha )
=w_2(\beta )$.
\end{thm}

\begin{proof} The commutative diagram
$$
\xymatrix{
\SO(4)\ar[d]^{(\rho_+,\rho_-)} \ar[r] \ar[rrd]^{\Lambda^2}
& \SU(4)\ar[r]^{\cong}\ar[rd]
& \Spin(6) \ar[d]\\
\SO(3)\times \SO(3) \ar[rr] &
& \SO(6)
}
$$           
implies that the double cover $\Spin(6)\to\SO(6)$ 
pulls back, under the inclusion of $\SO (3)\times \SO (3)$ in $\SO (6)$, 
to $(\rho_+,\rho_-):\SO (4) \to \SO (3)\times \SO (3)$.
Hence the obstruction to lifting from $\SO (3)\times\SO (3)$ to $\SO (4)$
is given by the sum of the second Stiefel-Whitney classes.
\end{proof}

\begin{rem}\label{Brauer} In \cite{DK} the orthogonal Brauer group
of a space $X$ was defined as the quotient of the monoid of all bundles
of simple central $\Rr$-algebras over $X$ (with a multiplication induced 
by the tensor product) by the submonoid of all bundles of 
the form $\End_\Rr(\mu)$ where $\mu$ is a real vector bundle over $X$. 
The inverse is given by the opposite algebra. Two bundles $\Ha$
and $\Hb$ are Morita equivalent if and only if they 
determine the same element of the Brauer group $\BO$.
There is a group isomorphism $$\BO\to H^0(X;\Zz/2)\oplus H^2(X;\Zz/2),$$
in which $\Ha$ corresponds to $(1,w_2(\alpha))$.
\end{rem}

If $\alpha$ and $\beta$ are Morita equivalent, we can 
describe all the morphisms from $\alpha$ to $\beta$ in $\mathcal{M}(X)$ from
the knowledge of one of them.

\begin{prop}\label{morphisms}
Given an isomorphism $\Ha\otimes\Hb^\op\cong\End_\Rr(\mu)$ and a real line
bundle $\delta$ we use the isomorphism $\End_\Rr(\mu)\to\End_\Rr
(\delta\otimes\mu)$ defined by the tensor product with the identity on $\delta$ 
to get $\Ha\otimes\Hb^\op\cong\End_\Rr(\delta\otimes\mu)$. This defines 
a bijection from $H^1(X;\Zz/2)$ to $\Hom_{\mathcal{M}(X)}(\beta,\alpha)$.
In particular, the automorphism group
$\operatorname{Aut}_{\mathcal{M}(X)}(\alpha)$ is isomorphic to
$H^1(X;\Zz/2)$.
\end{prop}

\begin{proof}
Isomorphisms $\Ha\otimes\Hb^\op\cong\End_\Rr(\mu)$ and 
$\Ha\otimes\Hb^\op\cong\End_\Rr(\mu')$ determine an isomorphism
$f:\End_\Rr(\mu)\to\End_\Rr(\mu')$. Since every automorphism of 
$\End_\Rr(\Rr^4)$ is inner, we can define $\delta$ as the line bundle
with the fibre at $x$ generated by an isomorphism $g:\mu_x\to\mu'_x$
such that $f_x(a)=gag^{-1}$. Then the map $g\otimes v\mapsto g(v)$ gives
an isomorphism $\delta\otimes\mu\to\mu'$. 
\end{proof}

\begin{cor}\label{a=b}
For a given $4$-dimensional oriented orthogonal vector bundle $\mu$
the bundles $\alpha=\rho_+(\mu)$ and $\beta=\rho_-(\mu)$ are isomorphic
if and only if $\mu$ has a $1$-dimensional (not necessarily trivial)
summand.
\end{cor}

With respect to the canonical orientations of $\alpha$ and $\beta$ 
given by $\mu$, this isomorphism can only be orientation reversing.
The existence of  a $1$-dimensional summand is equivalent
to the existence of an orientation reversing involution of $\mu$.

\begin{proof} If $\Ha\otimes\Ha^\op\cong\End_\Rr(\mu)$, then by the previous
proposition $\mu\cong\delta\otimes\Ha$ for some $\delta$ which is thus a
subbundle of $\mu$. Conversely, if an $\Ha$-bundle $\mu$ has a subbundle
$\delta$, then the multiplication gives an isomorphism $\mu\to\delta\otimes\Ha$ 
which means that $\rho_+(\mu)=\rho_+(\Ha)=\rho_-(\Ha)=\rho_-(\mu)$.  
\end{proof}

In Section 4 we will need

\begin{lem}\label{Euler}
Let $\alpha$, $\beta$, $\gamma$ be oriented 3-dimensonal vector bundles.
Suppose that $\Ha\otimes\Hb^\op\cong\End(\mu)$, $\Hb\otimes\Hh_{\gamma}^\op
\cong \End(\nu)$ so that $\Ha\otimes\Hh_{\gamma}^\op\cong\End(\mu\otimes_{\Hb}\nu)$. Give 
$\mu$, $\nu$ and $\mu\otimes_{\Hb}\nu$ the canonical orientations.
Then the Euler class of $\mu\otimes_{\Hb}\nu$ is 
$$e(\mu\otimes_{\Hb}\nu)=e(\mu)+e(\nu).$$
In other words, the Euler class gives a functor from $\mathcal{M}(X)$ to
the group $H^4(X;\Zz)$.
\end{lem}

\begin{proof}
As in the proof of Lemma \ref{p1} we use fibrewise classifying spaces. Let
$P\Hb (\Hb^{\infty})$ be the right quaternion projective bundle with the Hopf
bundle $H_1$ (a right $\Hb$-line bundle) and let $\Hb P(\Hb^{\infty})$ be the
left quaternion projective bundle with the left Hopf bundle $H_2$.
Then there are sections $s_1:X\to P\Hb (\Hb^{\infty})$ and $s_2:X\to \Hb
P(\Hb^{\infty})$ such that $\mu=s_1^*(H_1)$ and $\nu=s_2^*(H_2)$. So it is
enough to look at the fibre product $P\Hb(\Hb^{\infty})\times_X
\Hb P(\Hb^{\infty})$ and to show that $e(H_1\otimes_{\Hb} H_2)=e(H_1)+e(H_2)$.

Now $H^* (P\Hb (\Hb^{\infty})
\times_X \Hb P (\Hb^{\infty}))=H^*(X)[e_1,e_2]$,
where $e_i=e(H_i)$.
But $e(H_1\otimes H_2)=ae_1+be_2$ with $a,b\in\Zz$: it does not involve
any element of $H^*(X)$ since the restriction of $H_1\otimes_{\Hb} H_2$ to $X$ is 
$\Hb$ with the Euler class equal to zero. So one can calculate by 
restricting to the two factors $P \Hb (\Hb^{\infty})$ and $\Hb P(\Hb^{\infty})$ 
(using the inclusion of $X$ as $P\Hb(\Hb^1)$ or $\Hb P(\Hb^1)$).
The class restricts to $e_1$ and to $e_2$. Hence $a=b=1$.
\end{proof}

\section{$K$-theory}

We define $\KSp_\alpha^0(X)$ to be the Grothendieck group of
(left) $\Hh_\alpha$-bundles over the compact Hausdorff space $X$. The aim of
this section is to compute $\KSp_\alpha^0(X)$ as a classical $\KO$-group.
Let $X^{\alpha}$ stand for the Thom space of $\alpha$.

\begin{thm}\label{KSpalpha}
There is an isomorphism
$$
\KSp_\alpha^0 (X) \to \KO^0(\Hh_\alpha )=\tildeKO^{-1}(X^\alpha ) 
$$
given by mapping the class $[\xi ]$ to the element represented in
$\KO$-theory with compact supports by the linear map
$$
x\mapsto vx : \pi^*\xi \to \pi^*\xi
$$
over $v\in\Hh_\alpha$, where $\pi:\Ha\to X$ is the projection.
\end{thm}

\begin{rem}
The $K$-groups give a functor $\alpha\mapsto \KSp_\alpha^0(X)$
from $\mathcal{M}(X)$ to the category of abelian groups arising from
the functor $\mu\otimes_{\Hb}$ described in Proposition \ref{Morita_eq}.
So a Morita equivalence $\Ha\otimes\Hb^\op \to \End_{\Rr}(\mu)$
determines an isomorphism $\KSp^0_\beta (X) \to \KSp^0_\alpha (X)$,
which translates into the Bott isomorphism
$$
\KO^0(\Hh_\beta ) \to \KO^0(\Hh_\alpha )
$$
given by an associated spin structure for the virtual real vector
bundle $\Hh_\alpha -\Hh_\beta =\alpha -\beta$.
\end{rem}

We shall derive Theorem \ref{KSpalpha} from the following two propositions.
The first is the Karoubi-Segal periodicity theorem as given
in \cite{Z2}, Theorem 6.1. See also \cite{AD}, Theorem 3.3 and \cite{K},
pages 193--194. In the statement $L$ is the representation
$\Rr$ of $\Zz /2$ with the action of the generator as multiplication by $-1$.

\begin{prop}\label{KS}
Let $\zeta$ be a real vector bundle over $X$. Then there is an isomorphism
from the Grothendieck group $\KO_{C(\zeta )} (X)$
of graded $C(\zeta )$-modules, over the
Clifford algebra $C(\zeta )$ of a positive-definite inner product on
$\zeta$, to the $\Zz /2$-equivariant $\KO$-group of the total space
of the $\Zz /2$-equivariant vector bundle $L\otimes\zeta$:
$$
\KO_{C(\zeta )} (X) \to \KO_{\Zz /2}^0(L\otimes\zeta ),
$$
given by mapping the class of a graded (left) $C(\zeta )$-module 
$\mu=\mu_0\oplus\mu_1$ to the element represented by the linear map
$$
x\mapsto vx : \pi^*\mu_0 \to L\otimes\pi^*\mu_1
$$
over $v\in L\otimes \zeta$, where $\pi:\zeta\to X$ is the projection.
\end{prop}

Suppose that an orthogonal vector bundle $\zeta$ has dimension $4k$
and is oriented. Then we may define a central involution
$\omega_x\in C_0(\zeta_x)$ by $\omega_x = e_1\cdots e_{4k}$,
where $e_1,\ldots ,e_{4k}$ is any positively oriented orthonormal
basis of the fibre $\zeta_x$ at $x\in X$.
Then any graded $C(\zeta )$-module $\mu=\mu_0\oplus \mu_1$ splits as 
a direct sum of two graded submodules $\mu^+=\mu^+_0\oplus\mu^+_1$
and $\mu^-=\mu^-_0\oplus\mu^-_1$ such that $\omega_x$ acts as  identity 
in the fibres of $\mu^+_0$ and $\mu^-_1$, and as multiplication
by $-1$ in the fibres of $\mu^-_0$ and $\mu^+_1$. We will say that $\mu$ is {\it positive}
if $\mu =\mu^+$, and {\it negative}, if $\mu=\mu^-$. Then the Grothendieck
group of graded $C(\zeta)$-modules splits as a sum of the Grothendieck
groups of positive and negative modules $\KO^+_{C(\zeta)}(X)\oplus
\KO^-_{C(\zeta)}(X)$. The periodicity theorem (\cite{Z2}, Proposition 6.3
or \cite{Eu}, Proposition 3.1) gives an isomorphism between
$\KO^0_{\Zz/2}(L\otimes\zeta)$ and $\KO^0_{\Zz/2}(\zeta)$. This, combined 
with the Karoubi-Segal theorem,  establishes

\begin{prop}\label{KOplus}
Let $\zeta$ be an oriented real vector bundle over $X$ with dimension
a multiple of $4$. Then there is an isomorphism
from the Grothendieck group of positive $C(\zeta )$-modules to
the $\KO$-theory of the total space of $\zeta$
$$
\KO^+_{C(\zeta )}(X) \to \KO^0(\zeta ),
$$
given by mapping the positive $C(\zeta)$-module $\mu$ 
to the linear map
$$
x\mapsto vx : \pi^*\mu_0 \to \pi^*\mu_1
$$
over $v\in\zeta$.
\end{prop}

\begin{proof}[Proof of Theorem \ref{KSpalpha}] 
We apply the previous proposition with 
$\zeta =\Rr\oplus\alpha$. It will be convenient to name the generator of the 
first summand and write $\zeta =\Rr e\oplus\alpha$. We show that positive 
graded $C(\zeta)$-modules correspond to (ungraded) $\Hh_\alpha$-modules.

Indeed, let $\mu =\mu_0\oplus\mu_1$ be a positive graded 
$C(\Rr e\oplus\alpha)$-module.
Put $\xi =\mu_0$. The inclusion of $\Hh_\alpha =\Rr 1\oplus\alpha$ in 
$C_0(\Rr e\oplus\alpha)$
as $\Rr 1\oplus e\omega \alpha$ gives $\xi$ an $\Hh_\alpha$-structure.
Now $C_0(\Rr e\oplus\alpha ) = \Hh_\alpha \oplus \Hh_\alpha \omega$,
which, as a ring, is the product $\Hh_\alpha\times\Hh_\alpha$ (with the
factors corresponding to the idempotents $(1\pm\omega )/2$).
In the opposite direction, given an $\Ha$-bundle $\xi$, put $\mu_0=\xi$.
The inclusion $\Ha\subseteq C_0(\Rr e\oplus\alpha)$ described above 
and the action of $\omega$ as  the identity give $\mu_0$ the structure 
of a $C_0(\zeta)$-module, which extends uniquely to a positive graded 
$C(\zeta)$-module $\mu$.
\end{proof}

\section{Characteristic classes}

In this section we introduce characteristic classes for $\Ha$-bundles,
then describe their properties and relation to the Stiefel-Whitney, 
Euler and Pontryagin characteristic classes. 

Given an $\Ha$-module $\xi$ of dimension $n$ over $X$, we have an associated
projective bundle $\Ha P(\xi )$ over $X$
and a Hopf $\Ha$-line bundle $H$ (which we endow with the
canonical orientation). The following proposition
constructs characteristic classes $\da_i(\xi)$.

\begin{thm}\label{da}
For every $\Ha$-bundle $\xi$ of dimension $n$ there are uniquely determined
classes $\da_i(\xi)\in H^{4i}(X;\Zz)$, $1\le i\le n$ such that
the integral cohomology ring of $\Ha P(\xi)$ is given by
$$
H^*(\Hh_\alpha P(\xi);\Zz) =
H^*(X)[t]/
(t^n - \da_1(\xi )t^{n-1} + \cdots + (-1)^n \da_n(\xi )),
$$
where $t=e(H)\in H^4(\Hh_\alpha P(\xi );\Zz)$ is the Euler class 
of the Hopf bundle $H$.
\end{thm}

\begin{proof}
This follows at once from the Leray-Hirsch Theorem,
because the cohomology is freely generated as an $H^*(X;\Zz)$-module
by $1$, $t,\, \ldots ,\, t^{n-1}$.
\end{proof}                    

To derive the properties of the characteristic classes $\da_i$ we will use
a splitting principle for $\Ha$-bundles, which follows from Proposition \ref{da}
by induction, see \cite{MR1}.

\begin{prop}\label{splitting} 
For each $\Ha$-bundle $\xi$ over $X$ let $p:F(\xi)\to X$ be the bundle 
whose fibre at $x$ is the flag manifold of orthogonal splittings of $\xi_x$
as a sum $L_1\oplus L_2\oplus\dots\oplus L_n$ of $\Hh_{\alpha_x}$-lines.
Then $p^*(\xi)$ splits as a direct sum of $\Ha$-line bundles and 
$p^*:H^*(X;\Zz)\to H^*(F(\xi);\Zz)$ is injective.
\end{prop}

To shorten our notation put
$d^{\alpha}(\xi)=1+\da_1(\xi)+\da_2(\xi)+\dots+\da_n(\xi)$. 
We also write $\da_0(\xi )=1$.
The classes $\da(\xi)$ have 
the properties that one would expect and determine the other
characteristic classes of $\xi$ as a real vector bundle.

\begin{thm}\label{properties}
{\rm (a)} The classes $\da$ are multiplicative, i.e.
$\da(\xi\oplus\xi')=\da(\xi)\da(\xi')$ for $\Ha$-vector bundles $\xi$ and 
$\xi'$.
\newline
{\rm (b)} If $\xi$ is an $\Ha$-bundle of dimension $n$ with the canonical 
orientation, then its Euler class is $e(\xi)=\da_n(\xi)$.
\newline 
{\rm (c)} The Stiefel-Whitney classes of $\xi$ are
$$1+w_1(\xi)+w_2(\xi)+\dots+w_{4n}(\xi)=\sum_{i=0}^n
\left((1+w_2(\alpha)+w_3(\alpha)\right)^{n-i}\,\da_i(\xi).$$
In particular, $w_2(\xi)=nw_2(\alpha)$. 
\newline
{\rm (d)} 
The Chern classes of the complexification of $\xi$ {\rm (}and so the
Pontryagin classes $p_i(\xi )=(-1)^ic_{2i}(\Cc\otimes\xi )${\rm )} are given by
$$
\sum_{k=0}^{4n} c_k(\Cc\otimes\xi)
=\sum_{i,\,j =0}^n q_{n-i,n-j}(\alpha )\da_i(\xi)\da_j(\xi ),
$$
where the classes $q_{i,j}(\alpha )$ are determined by the
generating function
$$
\sum_{i,\, j\geq 0} q_{i,j}(\alpha )s^it^j =
\frac{(1-s)(1-t)+(e(\alpha )^2-p_1(\alpha))st}
{((1-s)^2+(e(\alpha )^2-p_1(\alpha))s^2) 
((1-t)^2+(e(\alpha )^2-p_1(\alpha))t^2)}.
$$
\end{thm}

\begin{proof}
(a) The proof is the same as in the case of the Stiefel-Whitney  and
Chern classes. 

(b) If $\mu$ is an $\Ha$-line, then $\da_1(\mu)=e(\mu)$ by definition. Now
using the multiplicativity and the splitting principle, we get 
$\da_n(\xi)=e(\xi)$.

(c) 
There is a quotient map $\pi :\Rr P(\xi ) \to \Hh_\alpha P(\xi )$
from the  real projective bundle to the $\Ha$-projective bundle. 
The Hopf $\Ha$-bundle
$H$ lifts to the tensor product $\Hh_\alpha\otimes H_\Rr$ with the real 
Hopf bundle $H_{\Rr}$.
So
$$
\pi^*(t)= x^4 + x^2w_2(\alpha ) +xw_3(\alpha ),
$$
where $x=w_1(H_\Rr )$. Comparing our definition of $\da$ with 
the corresponding
definition of the Stiefel-Whitney classes (or using the splitting principle), 
we get our formula.

(d) 
For an $\Ha$-line bundle $\mu$ we have
$$
c_1(\Cc\otimes\mu )=0,\quad
c_2(\Cc\otimes \mu )= 2 e(\mu)-p_1(\alpha ), \quad 
c_3(\Cc\otimes \mu )=e(\alpha )^2,\quad
c_4(\Cc\otimes\mu )= e(\mu )^2.
$$
From the splitting principle and multiplicativity,
we can express the total Chern class of $\Cc\otimes\xi$ as
$$
\prod_{j=1}^n
\left(y_j^2+2ay_j +b\right),
$$
where $\da_i(\xi)$ is the $i$-th elementary symmetric polynomial in
$y_1,\,y_2,\,\dots,\,y_n$, $a=1$ and $b=1+ e(\alpha )^2-p_1(\alpha)$. 
The stated formula is obtained by computing in the polynomial ring
$\Zz[a,b][y_1,u_2,\dots,y_n]$ on formal variables $a$, $b$, $y_j$ by
embedding $\Zz[a,b]$ as a subring of the polynomial ring $\Qq[r,s]$,
where $a=r+s$ and $b=rs$.
\end{proof}

\begin{cor}
If $n$ is even and $\xi$ admits an $\Hh_\alpha$-structure for some
$\alpha$, then $\xi$ is spin.
If $n$ is odd and $\xi$ admits an $\Hh_\alpha$-structure, then
$w_2(\xi )=w_2(\alpha )$.
\qed
\end{cor}

Now we examine the relation between the characteristic classes and 
Morita equivalence.

\begin{prop}\label{d+Me}
Let $\Ha\otimes\Hb^\op\cong\End_\Rr(\mu)$ and let $\eta$ be an $\Hb$-bundle. 
Then the characteristic classes of the $\Ha$-bundle $\xi=\mu\otimes_{\Hb}\eta$ 
are
$$
\sum_{i=0}^n \da_i(\xi) =
\sum_{i=0}^n (1+e(\mu ))^{n-i}d_i^\beta (\eta ).
$$
\end{prop}

\begin{proof} For an $\Hb$-line bundle $\eta$ this is Lemma
\ref{Euler}. Applying multiplicativity (or by using the equivalence between 
$\Ha P(\xi)$ and $\Ha P(\eta)$ under which the Hopf bundles correspond 
by tensoring with $\mu$), we get the formula.
\end{proof}

Using the Gysin exact sequence as in \cite{B} or in
\cite{C} one can describe the cohomology rings of $B\TSp(n)$. Denote by $\rho_2$
the reduction mod 2 and by $\Delta$ the corresponding Bockstein
homomorphism.
Let $\alpha$ and $\xi$ be the associated vector
bundles with fibres $\im\Hh$ and $\Hh^n$, respectively, to the classifying
$\TSp(n)$-principal bundle $E\TSp(n)$ over $B\TSp(n)$. Put $w_2=w_2(\alpha)$,
$p_1=p_1(\alpha)$ and $d_i=\da_i(\xi)$.

\begin{thm}\label{H*BTSp}
The cohomology rings of $B\TSp(n)$ are given by
\begin{align*}
H^*(B\TSp(n);\Zz/2)&=\Zz/2[w_2,Sq^1w_2,\rho_2d_1,\rho_2d_2,\dots,
\rho_2d_n],\\
H^*(B\TSp(n);\Zz)&=\Zz[\Delta w_2,p_1,d_1,d_2,\dots,d_n]/(2\Delta w_2).
\end{align*}
\end{thm}

The cohomology ring with $\Zz/2$ coefficients was described in \cite{MR2}
and \cite{MR3}.

\section{Complex quaternionic bundles}

In this section we will deal with  bundles of quaternion algebras which admit
as a subbundle a bundle of fields of complex numbers. Modules over them
are complex vector bundles where the complex structure extends to a 
quaternionic structure. We will characterize bundles of quaternion algebras which are 
Morita equivalent to such bundles of algebras (Theorem \ref{Hl_Hb}).  At the end we will derive 
a relation between the Chern classes and the classes $d_i^{\alpha}$ introduced 
above (Theorem \ref{Chern}).
 
We start with the description of bundles of fields of complex numbers. 
See \cite{CSS}.

\begin{Def}\label{Cd}
A principal $\Aut (\Cc )=\O (1)$-bundle  over $X$ determines an orthogonal real line
bundle $\delta$ and a bundle of complex algebras $\Cc_\delta =\Rr\oplus\delta$.
An element of the fibre $\delta_x$ with length $1$ is a square-root of
$-1$. We say that a real vector bundle $\xi$ over $X$ is a $\Cd$-bundle if
it has a $\Cd$-module structure.
\end{Def}

If $\xi$ has a $\Cd$-structure, then as in Remark \ref{metric}
there is a real inner product $g$ on $\xi$ such that the structure map
$\Cd\to\End(\xi)$ is a $*$-homomorphism, i.~e. $g(au,v)=g(u,\bar{a}v)$.
Let $\i$ be an element of $\delta_x$ of the length 1 and put
$f(u,v)=-\i g(\i u,v)$. The real bilinear form $f:\xi\times\xi\to\delta$ is
non-singular skew-symmetric and $\langle u,v \rangle=g(u,v)+f(u,v)$ is a $\Cd$-inner
product which is $\Cd$-linear in the first variable and $\Cd$-conjugate-linear 
in the second.

Conversely, if a $2n$-dimensional real vector bundle $\xi$ is equipped with
a non-singular skew-symmetric real bilinear form $\xi\times\xi\to\delta$,
one can prove that there is a $\Cd$-structure on $\xi$.  (See \cite{CSS}, 
Remark 5.5, or the proof of the similar Proposition \ref{bilin} below.)
 
Given a $2n$-dimensional vector space $V$ with a complex
structure, we can choose a compatible real inner
product on it. It enables us to introduce a subgroup
$$
\TU (V) =\{ g\in \O (V) \mid g(rv)=\kappa (r)g(v) \quad
\text{for some $\kappa\in\Aut (\Cc )$ and all $v\in V$}\}
$$
of $\O (V)$.
Thus, $\TU (V)$ consists of the $\Cc$-linear and the conjugate-linear
isometries.

As in the case of quaternionic structures, a $2n$-dimensional
orthogonal real vector bundle $\xi$ admits a $\Cc_\delta$-structure for some 
$\delta$ if and only if its structure group $\O(\Cc^n)$ can be reduced to 
$\TU (\Cc^n)$. 

Any $2$-dimensional real vector bundle $\lambda$ has a canonical
structure as a $\Cc_\delta$-line bundle, where $\delta =\det\lambda$.
Moreover, $\Cc_\delta$-line bundles over $X$ are classified by their
Euler class in the cohomology group $H^2(X;\, \Zz (\delta ))$ with
integral coefficients twisted by $\delta$. This also follows from the fact
that $\TU(\Cc)=\O(2)$.

An $n$-dimensional $\Cc_\delta$-bundle has twisted Chern classes
$c_j^\delta (\xi ) \in H^{2j}(X;\, \Zz (\delta^{\otimes j}))$ defined in the
same way as the classes $\da_j$ for an $\Ha$-bundle.
And these determine the Stiefel-Whitney classes of $\xi$:
in particular, $w_1(\xi )=nw_1(\delta )$ and $e(\xi )=c_n^\delta (\xi )$.

As in Section 3 one can show that the Grothendieck group $\K_\delta^0 (X)$ 
of $\Cc_\delta$-vector bundles is isomorphic to  
$\KO_{C(\Rr\oplus\delta)}(X)$ and, hence, to $\KO_{\Zz /2}^0(L\otimes\Cd)$, 
and then define $\KO^i_{\delta}(X)=\KO_{\Zz/2}^i(L\otimes\Cd)$
for $i\in\Zz$. See \cite{CSS}.

We now consider the situation in which a bundle of quaternions $\Hh_\alpha$
admits a bundle of fields $\Cc_\delta$ as a subalgebra.

\begin{prop}\label{Cd-Ha}
Let $\alpha$ be an $\SO (3)$-bundle. Then $\Hh_\alpha$ admits a subbundle
of the form $\Cc_\delta$ if and only if $\alpha = \delta\oplus\lambda$,
where $\lambda$ is a $2$-dimensional orthogonal real vector bundle and
$\delta
=\Lambda^2\lambda$ (or, in other words, if the structure group
of $\alpha$ can be reduced to the subgroup $\O (2)\subset\SO (3)$).
\end{prop}

\begin{proof} On $\Ha$ consider a real inner product in which the
multiplication by a given element from $\Ha$ is a $*$-homomorphism.
If $\Ha$ admits a $\Cd$ as a sub-algebra then the bundle 
$\lambda$ perpendicular to $\Cd$  is a $\Cd$-line
bundle. Hence $\delta=\Lambda^2\lambda$ and
$\alpha=\delta\oplus\lambda$.
\end{proof}

We will denote the bundle of quaternions
$\Hh_{\Lambda^2\lambda\oplus\lambda}$ determined by an $\O (2)$-bundle
$\lambda$ by simply $\Hh_\lambda$. The following propositions give two
necessary and sufficient conditions for a vector bundle to have
$\Hl$-structure.

\begin{prop}\label{bilin}
Fix a $\Cc_\delta$-line bundle $\lambda$.
Let $\xi$ be a $\Cc_\delta$-vector bundle.
Then there is a natural correspondence,
up to homotopy, between $\Hh_\lambda$-structures on $\xi$ and
non-singular skew-symmetric $\Cc_\delta$-bilinear forms
$\xi\otimes_{\Cc_\delta} \xi \to \lambda$.
\end{prop}

\begin{proof} We will carry out all the constructions in fibres at a given
point $x\in X$. Suppose that $\xi$ has an $\Hl$-structure. Then $\xi$ can be 
equipped with a real inner product from Remark \ref{metric} which is the first
component of a $\Cd$-inner product as shown below Definition \ref{Cd}. For every 
$a\in \lambda_x$ the multiplication $\varphi_a(v)=a\cdot v$
defines a $\Cd$-conjugate linear map $\xi\to\xi$ such that 
$\varphi_a^2(v)=-|a|^2v$, where $|a|$ is the norm given by the  real inner 
product on $\Hl$. For the adjoint we get
$\varphi_a^*=-\varphi_a$. In the first place $*$ means the adjoint with
respect to the real inner product, but this translates into the adjoint for
the $\Cd$-inner product. (For a conjugate-linear map $\varphi$, the adjoint
is defined so that $\langle \varphi(u),v\rangle=\langle
\varphi^*(v),u\rangle=\overline{\langle u,\varphi^*(v)\rangle}$.) So we can
define  a skew-symmetric bilinear
form $f:\xi\otimes_{\Cd}\xi\to\lambda$ by 
$$\langle f(u,v), a \rangle=\langle u,\varphi_a(v)\rangle$$
using the $\Cd$-inner product. $f$ is non-singular since $\varphi_a$ is
non-singular for $a\ne 0$.

Conversely, given  $f$, we define a $\Cd$-conjugate linear map 
$\psi_a:\xi_x\to\xi_x$ by 
$\langle u,\psi_a(v)\rangle=\langle f(u,v),a\rangle$. Then
$\psi_a^*=-\psi_a$ and $\psi_a\circ\psi_a^*$ is $\Cd$-linear and positive 
definite for $a\ne 0$. So we can define 
$$\varphi_a=|a|(\psi_a\psi_a^*)^{-1/2}\psi_a$$
with $\varphi_a^2=-|a|^2$ and $\varphi_a^*=-\varphi_a$.

These two constructions define maps from $\Hl$-structures to non-singular
skew-symmetric bilinear forms and back. Composing in one direction we get
the identity on $\Hl$-structures. In the other direction we get a homotopic
form, through the homotopy $(|a|(\psi_a\psi_a^*)^{-1/2})^t$, $0\le t\le
1$.
\end{proof}

\begin{prop}\label{Hl-bundle} Let $\xi$ be a $4n$-dimensional orthogonal 
real vector bundle. Then $\xi$ admits an $\Hh_\lambda$-structure for some 
complex line bundle $\lambda$ if and only if the structure group of $\xi$
can be reduced from $\SO(\Hh^n)$ to $\TSp(\Hh^n)\cap \U (\Hh^n)$,
and an $\Hh_\lambda$-structure for some $\Cc_\delta$-line bundle $\lambda$
and some $\delta$ if and only if the structure group reduces to
$\TSp (\Hh^n)\cap \TU (\Hh^n)$.
\end{prop}

\begin{proof} Use Lemma \ref{TSp} and the analogous statement for 
$\Cd$-structures.
\end{proof}

It is advantageous to express the intersections of the groups 
as quotients of products. 

\begin{lem} \label{groups} For any $4n$-dimensional real vector space $V$ 
with a left $\Hh$-module structure
\begin{align*}
\TSp(V)\cap \U(V)&=\U(1)\cdot \Sp(V)=\left( \U(1)\times \Sp(V))\right) /(\Zz/2),\\
\TSp(V)\cap \TU(V)&=\TU(1)\cdot \Sp(V)=\left( \TU(1)\times \Sp(V))\right)
/(\Zz/2).
\end{align*}
\end{lem}

\begin{proof} Let us prove only the second formula. Consider an element of
$\TSp(V)$ given by an isomorphism $v\mapsto ag(v)$ where $a\in \Sp(1)$ and
$g\in\Sp(V)$. This element lies in $\TU(V)$ if and only if it is a $\Cc$-linear 
or  conjugate linear isometry. Since $g$ is $\Cc$-linear, this means that for all 
$z\in\Cc$ either $az=za$ or $az=\bar{z}a$. Hence $a\in \U(1)\cup \j
\U(1)=\TU(1)\subset \Sp(1)$.
\end{proof}

Now we return to the construction from the beginning of Section 2.
A $2$-dimensional complex bundle $\mu$ or, more generally,
a $2$-dimensional $\Cc_\delta$-vector bundle has a natural orientation
as a $4$-dimensional real vector bundle.
The double cover $\SO (\Hh ) \xrightarrow{(\rho_+,\,\rho_-)} \SO (3)\times \SO (3)$
considered in Proposition \ref{Ha_Hb} restricts to  maps
$$
\U (2) =\U (\Hh ) \to \SO (3)\times \SO (3),\quad 
\TU (2)=\TU(\Hh)\to\SO(3)\times\SO(3).
$$
Since $\TSp(\Hh)=\SO(\Hh)$, we can apply Lemma \ref{groups} to get
$$\U(2)\cong \U(1)\cdot \Sp(1)=\Spin^c(3), \quad
\TU(2)\cong\TU(1)\cdot\Sp(1).$$

\begin{lem}\label{TU2}
Let $\mu$ be a $2$-dimensional $\Cd$-vector bundle.
Then $\mu$ is an $\Hh_\lambda$-line bundle, where $\lambda =
\Lambda^2_{\Cc_{\delta}} \mu$,
and also an $\End_{\Hl} (\mu )$-line bundle.
\end{lem}

\begin{proof}
It follows from Proposition \ref{bilin} that $\mu$ has
an $\Hl$-structure. Then 
$\Hh_{\rho_-(\mu)}^\op=\End_{\Hl}(\mu)$ by Proposition \ref{Ha_Hb}
and $\rho_-(\mu)$ is the bundle of skew-Hermitian endomorphisms 
of $\mu$.
\end{proof}

The homomorphism $\rho_-$ restricted to $\TU(2)\cong\TU(1)\cdot \Sp(1)$ is 
the projection onto $\SO(3)$. We prove that the structure group 
$\SO(3)$ of an oriented $3$-dimensional real vector bundle $\beta$ 
can be lifted to $\TU(2)$ if and only if there are a real line bundle 
$\delta$ and an element $l\in H^2(X;\Zz(\delta))$ such that 
$w_2(\beta)=\rho_2l+w_1^2(\delta)$.  

The condition is necessary. According to Lemma \ref{TU2} for 
a $2$-dimensional $\Cd$-vector bundle $\mu$ with
$\rho_-(\mu)=\beta$ we have $\rho_+(\mu)=\delta\oplus\lambda$ and
Proposition \ref{Ha_Hb} implies that
$w_2(\beta)=w_2(\rho_+(\mu))=\rho_2 e(\lambda)+w_1^2(\delta)$.

To show that the condition is sufficient take the $\Cd$-line bundle $\lambda$ 
with the Euler class $l$. Then by Proposition \ref{w2} the quaternion
bundles $\Hl$ and $\Hb$ are Morita equivalent and the vector bundle $\mu$ 
from the definition of the equivalence has the structure group $\TU(2)$. 
So we obtain

\begin{lem} \label{TSpinc}
Let $\beta$ be an oriented $3$-dimensional vector bundle with
$w_2(\beta)=\rho_2(l)+w_1^2(\delta)$ for an element $l\in H^2(X;\Zz(\delta))$ and a real
line bundle $\delta$. Then there is a $2$-dimensional $\Cd$-vector bundle 
$\mu$ such
that
$\operatorname{det}_{\Cd}\mu=\lambda$, where $e(\lambda)=l$, and
$$\Hb=\End_{\Hl}(\mu)^\op.$$
\end{lem}

For $\delta$ trivial, if $\beta$ is a $3$-dimensional $\Spinc$ real vector
bundle,
then $\Hb$ is Morita equivalent to an $\Hl$  for some complex line
bundle $\lambda$.

Applying Proposition \ref{Morita_eq} we obtain: 

\begin{thm}\label{Hl_Hb} 
Let $\eta$ be a $4n$-dimensional real vector bundle. Then $\eta$ admits
an $\Hb$-structure with $w_2(\beta)=\rho_2(l)+w_1^2(\delta)$ for an 
$l\in H^2(X;\Zz(\delta))$ and a real line
bundle $\delta$ if and only if for the $\Cd$-line bundle $\lambda$ with
$e(\lambda)=l$ there exists an $\Hl$-line bundle $\mu$ and an 
$n$-dimensional $\Hl$-vector bundle $\xi$ such that
$$\eta\cong \mu^*\otimes_{\Hl}\xi=\Hom_{\Hl}(\mu,\xi).$$
The twisted quaternionic $\Hb$-structure is given by the action of the
bundle $\End_{\Hl}(\mu)^\op$.
\end{thm}

The following statement is a complement to Proposition \ref{Ha_Hb}.

\begin{lem}\label{split}
Suppose that the oriented orthogonal $4$-dimensional bundle $\mu$
is an orthogonal direct sum $\mu_0\oplus\mu_1$ of two $2$-dimensional
subbundles. Write $\delta =\det \mu_0=\det \mu_1$, so that $\mu_0$ and
$\mu_1$ become $\Cd$-bundles. Then
$\rho_+(\mu)= \delta\oplus (\mu_0\otimes_{\Cc_\delta} \mu_1)$ and
$\rho_-(\mu) = \delta\oplus (\mu_0\otimes_{\Cc_\delta} \bar{\mu_1})$, where the
conjugate is given by conjugation on $\Cc_\delta$.
\end{lem}

\begin{proof}
The vector bundle $\mu=\mu_0\oplus\mu_1$ has the structure
group $O(2)\cdot O(2)=TU(1)\cdot TU(1)\subset TU(2)$. So it follows from 
Lemma \ref{TU2} that $\rho_+(\mu)$ is given by the projection 
onto the first factor $TU(1)$ and  is equal to $\delta\oplus\Lambda^2_{\Cd}(\mu_0\oplus
\mu_1)=\delta\oplus(\mu_0\otimes_{\Cd}\mu_1)$.

Since the orientation of $\mu_0\oplus\overline{\mu_1}$ is opposite to that of $\mu$,
$$\rho_-(\mu)=\rho_+(\mu_0\oplus\overline{\mu_1})=\delta\oplus(\mu_0
\otimes_{\Cd}\overline{\mu_1}).$$
\end{proof}

In the case that $\lambda$ is a $\Cd$-line bundle,
and $\alpha =\delta\oplus\lambda$ we can express the Chern classes $\cd_i$
in terms of the  classes $d_j^\alpha$.

\begin{thm}\label{Chern} 
Let $\alpha=\delta\oplus\lambda$, where $\lambda$ is a
$\Cd$-line bundle. Let $\xi$ be an $\Ha$-bundle of dimension $n$. Then
$$1+c_1^{\delta}(\xi)+\dots+\cd_{2n}(\xi)=
\sum_{i=0}^n(1+\cd_1(\lambda))^{n-i}\da_i(\xi).$$
In particular, $\cd_1(\xi)=n\cd_1(\lambda)$.
\end{thm}

\begin{proof} We use the splitting principle for $\Ha$-vector bundles
and multiplicativity of the Chern classes $\cd=1+\cd_1+\cd_2+\dots$ and the
characteristic classes $\da$. Thus it is sufficient to carry out 
the proof only for $n=1$.
Let $\mu$ be an $\Hl$-line, where $\lambda$ is a $\Cd$-line. Then 
$\cd_1(\mu)=\cd_1(\lambda)$, since we may assume that $X$ is 
a 3-dimensional CW-complex, over which $\mu=\Cd\oplus\lambda$. Further,
$\cd_2(\mu)=e(\mu)=\da_1(\mu)$. Consequently,
$$1+\cd_1(\mu)+\cd_2(\mu)=1+\cd_1(\lambda)+\da_1(\mu).$$
\end{proof}

\begin{rem}\label{twist2}
Consider an $\SO(3)$-bundle $\alpha$ and an $ \Ha$-bundle $\xi$ over
$X$.
We can lift to the sphere bundle of $\alpha$ by $\pi : S(\alpha )\to X$.
Over $S(\alpha)$ we have a complex line bundle $\lambda$
such that $\pi^*\alpha = \Rr\oplus\lambda$.  So $\pi^*\xi$ is an
$\Hl$-bundle.
In particular, $\pi^*\xi$ is complex and has Chern classes in 
$H^*(S(\alpha))$.
Since the Euler class of $\alpha$ is $2$-torsion, we have a rational
splitting:  $H^{2i}(S(\alpha ); \Qq) =
H^{2i}(X; \Qq) \oplus H^{2(i-1)}(X; \Qq)$. More precisely, $H^*(S(\alpha);    
\Qq)$ is an $H^*(X; \Qq)$-module with a generator $s\in H^2(S(\alpha);\Qq)$  
subject to a relation $s^2+as+b=0$ for some elements $a\in H^2(X;\Qq)$ and
$b\in H^4(X;\Qq)$. If $e(\alpha)=0$, this is true also over $\Zz$.

Now using Theorem \ref{Chern} 
and naturality of the classes $d_i$, the Chern classes of $\pi^*\xi$ are
$$1+c_1(\pi^*\xi)+\dots +c_{2n}(\pi^*\xi)=
\sum_{i=0}^n (1+c_1(\lambda))^{n-1}\pi_*d_i^{\alpha}(\xi).$$

If $X$ is an almost quaternionic smooth manifold as in Remark \ref{twist1},
then the tangent bundle of $S(\alpha)$ being isomorphic to 
$\lambda\oplus \pi^*\tau X$ has a complex structure, i.e. the twistor space 
$S(\alpha)$ is almost complex, which is well known in quaternionic geometry.
\end{rem}

\section{Complexified quaternionic bundles}

Given a real line bundle $\delta$ and an oriented $3$-dimensional vector
bundle $\alpha$ (with inner product) we may consider the bundle 
of algebras $\Ha\otimes\Cd$. It
depends only on the $3$-dimensional vector bundle $\delta\otimes\alpha$.
Indeed, it may be identified with the (ungraded) Clifford algebra bundle
$C(\alpha\otimes\delta)$ with the positive-definite quadratic form. 
Locally, if $e_1$, $e_2$, $e_3$ is an orthonormal basis of
$\alpha\otimes\delta$ then the corresponding fibre of
$C(\alpha\otimes\delta)$ is generated by $e_1$, $e_2$, $e_3$ with $e_i^2=1$
and $e_ie_j=-e_je_i$ for $i\ne j$, the fibre of $\Cd$ is $\Rr 1\oplus \Rr
(e_1e_2e_3)$, and the fibre of $\Ha$ is $\Rr 1\oplus \Rr (e_1e_2, e_2e_3,
e_1e_3)$. The centre (formed in each fibre) is $\Cd$.

There is an $\Rr$-algebra isomorphism $\Hh\otimes\Cc\to \End_{\Cc}(\Hh)$,
that is $M_2(\Cc)$, under which the group of automorphisms
$\Aut_\Rr(\Hh)\times\Aut_\Rr(\Cc)$ maps to the retract $\SU(2)/\{\pm 1\}\rtimes
\Aut_{\Rr}(\Cc)$ of the full automorphism group $\GL(2)/\Cc^*\rtimes \Aut_{\Rr}(\Cc)$
of $M_2(\Cc)$.

Consequently, we can describe bundles of algebras $\Ha\otimes\Cd$ as just those
$\Rr$-algebra bundles with fibres of type $\Hh\otimes\Cc$ which
have the structure group $\Aut_\Rr(\Hh)\times\Aut_\Rr(\Cc)=\SO(3)\times\O(1)
=\O(3)$.

Consider a $4$-dimensional $\Cd$-vector bundle $\mu$ which is a left
$\Ha\otimes\Cd$-bundle. This means that the fibre  is isomorphic to 
the complex vector space $\Hh\otimes\Cc$ with the obvious action of the algebra
$\Hh\otimes\Cc$. On $\mu$ we can choose a real  inner 
product such that the
action $\Ha\otimes\Cd\to\End_{\Cd}(\mu)$ is a $*$-homomorphism. 
We will show that $\mu$ is associated  to a principal bundle 
with the structure group $\TSp(\Hh)\cdot\TU(\Cc)
=\SO(\Hh)\cdot \TU(\Cc)\subset\TU(\Hh\otimes\Cc)$. In traditional terminology
this is the group
$(\SO(4)\times \TU(1))/\{(\pm 1\}=\SO(4)\cdot\TU(1)\subset \TU(4)$.
Put
\begin{align*}
\operatorname{Fr}(\mu)=\{&f\in\Hom_\Rr(\Hh\otimes\Cc,\mu) \mid
\text{$f$ preserves the inner product, $f(rv)=\kappa(r)f(v)$}\\
&\text{for some $\kappa\in \Iso(\Hh,\Ha)\times\Iso(\Cc,\Cd)$ and all 
$v\in\Hh\otimes\Cc$, $r\in \Hh\otimes\Cc$}\}.
\end{align*}
Then $P=\operatorname{Fr}(\mu)$ is a principal bundle with the structure
group $\TSp(\Hh)\cdot\TU(\Cc)$ such that
$\mu=P\times_{\TSp(\Hh)\cdot\TU(\Cc)}(\Hh\otimes\Cc)$ and
$\Ha\otimes\Cd=P\times_{p}(\Hh\otimes\Cc)$, where $p:\TSp(\Hh)\cdot\TU(\Cc)
\to\Aut_{\Rr}(\Hh)\times\Aut_\Rr(\Cc)$ is the projection. This proves

\begin{prop}\label{so4tu1}
A $4$-dimensional $\Cd$-vector bundle has an $\Ha\otimes\Cd$-module
structure if and only if its structure group $\TU(4)$ can be reduced to
$\SO(4)\cdot\TU(1)$.
\end{prop} 

Now we can proceed as in Section 2. In many of the arguments $\SO(4)\cdot\TU(1)$
plays the role taken there by $\SO(4)$.
 
Since $\TU(1)$ is a semidirect product of groups $\U(1)$ and $\Zz/2$, its
elements can be described by pairs $(c,s)\in \U(1)\times
\Zz/2$ having the action on $\Cc$ given by $(c,1)z=cz$ and $(c,-1)z=c\bar{z}$.
Consider the  covers
$$\Sp(1)\times \Sp(1)\times\TU(1)\to \SO(4)\cdot\TU(1)
\xrightarrow{(\rho_+,\,\rho_-,\,\rho_0)} \SO(3)\times\SO(3)
\times\TU(1)$$
given by mapping $(a,b,(c,s))\in \Sp(1)\times\Sp(1)\times\TU(1)$ to the map
$g:\Hh\otimes\Cc\to \Hh\otimes\Cc:\, v\otimes z\mapsto av\bar{b}\otimes
(c,s)z$
and $g$ to $(\rho_+(g),\rho_-(g),\rho_0(g))=(\rho(a),\rho(b),z\mapsto
(c^2,s)z)$. Further, let $\tau:\SO(4)\cdot\TU(1)\to \Aut(\Cc)=\Zz/2$
be the map $\tau(g)=s$.

\begin{prop}\label{Ha_Hb_Cd}
Let $\mu$ be a $4$-dimensional $\Cd$-vector bundle over $X$ the structure
group of which can be reduced to $\SO(4)\cdot\TU(1)$. Put
$\alpha=\rho_+(\mu)$, $\beta=\rho_-(\mu)$, $\sigma=\rho_0(\mu)$.
Then $\Cd=\tau(\mu)$, the vector bundle $\mu$ is a left
$\Ha\otimes\Cd$-module and
a right $\Hb\otimes\Cd$-module, $\sigma$ is a $\Cd$-line bundle and
there is a canonical isomorphism of bundles of $\Cd$-algebras
$$(\Ha\otimes\Cd)\otimes_{\Cd}(\Hb\otimes\Cd)^\op=
(\Ha\otimes\Hb^\op)\otimes\Cd\cong\End_{\Cd}(\mu).$$
Moreover, there are canonical isomorphisms of vector bundles
$$(\alpha\oplus\beta)\otimes\sigma\cong\Lambda^2_{\Cd}\mu
\quad\text{and}\quad \bar{\sigma}^{\otimes 2}\otimes_{\Cd}\Lambda^4\mu
\cong\Cd$$
where $\bar{\sigma}$ is the conjugate vector bundle to $\sigma$.

Conversely, if $\alpha$ is an oriented orthogonal $3$-dimensional vector
bundle and $\mu$ is an $\Hh_\alpha\otimes\Cd$- bundle of dimension $4$ 
over $\Cd$, then the real
form of $\bar{\sigma}^{\otimes 2}\otimes_{\Cd}\Lambda^4\mu$ acquires
an orientation under which $\rho_+(\mu )$ is identified with
$\alpha$, $\tau(\mu)$ is identified with $\Cd$ and $\beta =\rho_-(\mu )$ is characterized
by an isomorphism (of bundles of algebras)
$$
(\Hh_\beta\otimes\Cd)^\op = \End_{\Hh_\alpha\otimes\Cd} (\mu ).
$$

\end{prop}

\begin{proof} The $4$-dimensional $\Cc$-vector space $\Hh\otimes\Cc$ is
the complexification of the real vector space $\Hh$. Now we can carry out the proof
by complexifying all the vector spaces $\im\Hh$, $\End_{\Rr}(\Hh)$,
$\Hh\otimes\Hh^{\op}$, $\Lambda^2\Hh$, $\Lambda^4\Hh$ and
the homomorphisms used in the proof of Proposion \ref{Ha_Hb} and
Lemma \ref{lambda2}, by checking
that
these complexified homomorphisms are invariant with respect to appropriate
actions of $\SO(4)\cdot\TU(1)$ and by writing $\alpha$, $\beta$, $\Cd$,
$\sigma$,
$\Lambda_{\Cd}^2\mu$ and $\Lambda^4_{\Cd}$ as vector bundles
associated to a $\SO(4)\cdot\TU(1)$-principal bundle determined by $\mu$. 
Details are left to the reader.
\end{proof}

\begin{rem}\label{char_cl} The isomorphism
$\Lambda_{\Cd}^4\mu\cong\sigma^{\otimes 2}$ implies that
$\cd_1(\mu)=2\cd_1(\sigma)$, and from the isomorphism
$(\alpha\oplus\beta)\otimes\sigma\cong\Lambda_{\Cd}^2\mu$ we get
$2\cd_2(\mu)=-p_1(\alpha)-p_1(\beta)+3(\cd_1(\sigma))^2$.
\end{rem}

\begin{prop} \label{abs}
Given $3$-dimensional oriented vector bundles $\alpha$, $\beta$ and 
a $\Cd$-vector bundle $\sigma$, there is a $4$-dimensional vector bundle
$\mu$ such that $\rho_+(\mu)=\alpha$, $\rho_-(\mu)=\beta$ and
$\rho_0(\mu)=\sigma$ if and only if
$w_2(\sigma)=w_2(\alpha)+w_2(\beta)$.
\end{prop}

\begin{proof} The commutative diagram
$$
\xymatrix{
\SO(4)\cdot\TU(1)\ar[d]_{(\rho_+,\,\rho_-,\,\rho_0)} \ar[r] \ar[rrd]
 & \SU(4)\cdot\TU(1)\ar[r]^{\cong}\ar[rd]^{\Lambda^2_{\Cc}}
& \Spin(6)\cdot\TU(1) \ar[d]\\
\SO(3)\times \SO(3)\times \TU(1) \ar[rr] &
& \SO(6)\cdot\TU(1)
}
$$
implies that the double cover $\Spin(6)\cdot\TU(1)\to\SO(6)\cdot\TU(1)$
pulls back, under the homomorphism $(a,b,c)\mapsto (a+b)\otimes c$,
to $(\rho_+,\rho_-, \rho_0)$.
Hence the obstruction to lifting from $\SO (3)\times\SO (3)\times\TU(1)$ 
to $\SO (4)\cdot\TU(1)$
is given by $w_2(\sigma)=w_2(\alpha)+w_2(\beta)$.
\end{proof}

\begin{Def}
For each $\delta$ we define a Morita category $\mathcal{M}_\delta(X)$ with objects the
oriented $3$-dimensional orthogonal vector bundles $\alpha$, $\beta$ and morphisms
$\beta\to\alpha$ given by a $\Cd$-isomorphism of algebras
$$(\Ha\otimes\Cd)\otimes_{\Cd}(\Hb\otimes\Cd)^\op=(\Ha\otimes\Hb^\op)\otimes\Cd
\to\End_{\Cd}(\mu),$$
where $\mu$ is a $4$-dimensional $\Cd$-vector bundle, up to isomorphisms of
$\mu$.
\end{Def}

Given any
$\Cd$-line bundle $\lambda$, we use the isomorphism $\End_{\Cd}(\mu)=
\End_{\Cd}(\lambda\otimes_{\Cd}\mu)$ to get an action of the Piccard group
$\operatorname{Pic}_\delta(X)=H^2(X;\Zz(\delta))$ on
$\Hom_{\mathcal{M}_\delta(X)}(\beta,\alpha)$, and then, as in
Proposition \ref{morphisms} we have
$\Aut_{\mathcal{M}_\delta(X)}(\alpha)=\operatorname{Pic}_\delta(X)$.

Proposition \ref{abs} implies immediately

\begin{thm}\label{Me_Cd}
There is a morphism from $\beta$ to $\alpha$ in $\mathcal{M}_\delta(X)$
if and only if
$e(\alpha\otimes\delta)=e(\beta\otimes\delta)$.
\end{thm}

\begin{proof}
Denote the Bockstein homomorphism corresponding to the exact sequence
$0\to\Zz(\delta)\to\Zz(\delta)\to \Zz/2\to 0$ by $\Delta_{\delta}$.
A $\Cd$-line bundle $\sigma$ such that $w_2(\sigma)=
w_2(\alpha)+w_2(\beta)$ exists if and only if $\Dd w_2(\alpha)=\Dd
w_2(\beta)$
and this is equivalent to our condition, since $e(\alpha\otimes\delta)=
\Dd(w_2(\alpha)+w_1^2(\delta))$.
\end{proof}

\begin{prop}\label{modules}
Any isomorphism $\Ha\otimes\Hb^{\op}\otimes\Cd\cong\End_{\Cd}(\mu)$
determines an equivalence from the category of $\Hb\otimes\Cd$-modules
to the category of $\Ha\otimes\Cd$-modules given by
$$\eta\mapsto\mu\otimes_{\Hb\otimes\Cd}\eta.$$
\end{prop}

As a consequence we get

\begin{cor}\label{gen}
Let $e(\alpha\otimes\delta)=0$. Then there is an
$\Ha\otimes\Cd$-bundle $\omega$ of $\Cd$-dimension $2$ such that
any $\Ha\otimes\Cd$-bundle of $\Cd$-dimension $2n$ is of the form 
$\omega\otimes_{\Cd}\zeta$
for a $\Cd$-bundle $\zeta$ of dimension $n$.
\end{cor}

\begin{proof} Just as $\Hh\otimes\Cc=\End_\Cc(\Hh)$, we have
$\Hh_{\Cd}\otimes\Cd=
\End_{\Cd}(\Hh_{\Cd})$.  Every
$\Hh_{\Cd}\otimes\Cd$-bundle $\eta$ is therefore of the form
$\Hh_{\Cd}\otimes_{\Cd}\zeta$, where 
$\zeta=\Hom_{\Hh_{\Cd}\otimes\Cd}(\Hh_{\Cd},\eta)$. 

Since $e((\delta\oplus\Cd)\otimes\delta)=0=e(\alpha\otimes\delta)$, there is
an isomorphism from $\delta\oplus\Cd$ to $\alpha$ in
$\mathcal{M}_{\delta}(X)$ represented by a $4$-dimensional vector bundle
$\mu$. Now any $\Ha\otimes\Cd$-bundle is of the form
$$\mu\otimes_{\Hh_{\Cd}\otimes\Cd}(\Hh_{\Cd}\otimes_{\Cd}\zeta)=
(\mu\otimes_{\Hh_{\Cd}\otimes\Cd}\Hh_{\Cd})\otimes_{\Cd}\zeta.$$
Putting $\omega=\mu\otimes_{\Hh_{\Cd}\otimes\Cd}\Hh_{\Cd}$, we get the 
assertion. 
\end{proof}

\begin{rem}
Consider a compact Hausdorff space $Y$ with an involution. Let $E$ be a complex vector bundle
over $Y$ and let $J:\xi\to\xi$ be a conjugate-linear map 
lifting the involution on $Y$ such that $J^2=-1$. The pair $(E,J)$ is variously
called a quaternionic or symplectic bundle over $Y$, \cite{D, S, dSLF}.
We relate this notion to complexified quaternionic bundles in our sence.

Let $\Cc P(\Hh)$ be the complex projective space modelled on
the $2$-dimensional complex vector space $\Hh$ with the involution given by
multiplication by $\j$. Define $X$ to be the quotient of $Y\times\Cc P(\Hh)$
by  the free involution. Let $\delta$ be the real line bundle over $X$
given by the double covering $p:Y\times \Cc P(\Hh)\to X$. We associate with $E$ 
a vector bundle $\eta$ such that the fibre over $x\in X$ is
$\eta_x=E_y\oplus E_{y'}$ where $p^{-1}(x)=\{y,y'\}$. 
This bundle  has an $\Hh\otimes\Cd$-structure. The multiplication by $\j\in
\Hh$ is given by $\j(u,v)=(Jv,Ju)$. The $\Cd$-structure is defined
as follows. Let $t$ be the involution
$(-1,1):E_y\oplus E_{y'}=\eta_x\to\eta_x$, and $t'=-t$. Then
$t\i=\i t$ and $t\j=-\j t$. So $(\i t)\i=\i(\i t)$ and 
$(\i t)\j=\j(\i t)$, and $(\i t)^2=-1$. So we can use $\i t=-\i t'$ to
define the $\Cd$-structure commuting with the $\Hh$-multiplication. 

In the same way the complex Hopf bundle $H$
over $\Cc P(\Hh)$ determines an $\Hh\otimes\Cd$-vector bundle $\omega$ over $X$
of complex dimension $2$ (where $\omega_x=H_y\oplus H_{y'}$). 
As in Corollary \ref{gen} there
is a $\Cd$-bundle $\zeta=\Hom_{\Hh\otimes\Cd}(\omega,\eta)$ over $X$ such 
that $\eta=\omega\otimes_{\Cd}\zeta$.
Its lift to $Y\times\Cc P(\Hh)$ is $E\otimes_{\Cc} H$.
\end{rem}

\begin{rem}\label{BrU}
One can define  a Brauer group
$\operatorname{BrU}_{\delta}(X)$ of central simple $\Cd$-algebras
and show that it is isomorphic to $\operatorname{Tor}H^3(X;\Zz(\delta))$.
The class of $\Ha\otimes\Cd=C(\alpha\otimes\delta)$ is 
$e(\alpha\otimes\delta)$.
\end{rem}

There is a complex $\K$-theory of $\Ha\otimes\Cd$-vector bundles modelled on
the real $\K$-theory of Section 3.

\begin{prop}\label{Kdh}
There is an isomorphism from the Grothendieck group of \ $\Ha\otimes\Cd$-modules over $X$
to $\K^0_{\delta}(\Ha)$, the $\K_\delta$-theory with compact supports, given by
mapping the class $[\xi]$ to the element represented by the linear map
$$x\mapsto vx:\pi^*\xi\to\pi^*\xi$$
over $x\in\Ha$, where $\pi:\Ha\to X$ is the projection. 
\end{prop}

\end{document}